\makeatletter

\documentclass[12pt]%
{article}

\newsavebox{\savepar}

\usepackage{amsmath}
\usepackage{amsfonts}
\usepackage{euscript}
\usepackage{oldgerm}
\usepackage{amsthm}
\usepackage{rotating}
\usepackage{mathrsfs}
\usepackage{hyperref}
\usepackage{amssymb}
\usepackage{epsfig}
\usepackage{dsfont}
\usepackage{subfig}
\makeindex

  \newtheorem{theorem}{Theorem}[section]

  \theoremstyle{definition}
\newtheorem{definition}[theorem]{Definition}

  \theoremstyle{remark}

\newtheorem{lemma}[theorem]{Lemma}
\newtheorem{proposition}[theorem]{\bf Proposition}
\newtheorem{Corollary}[theorem]{\bf Corollary}
\theoremstyle{definition}

\theoremstyle{remark}

\begin{document}

\newcommand{\norm}[1]{\left\lVert #1\right\rVert}
\newcommand{\namelistlabel}[1]{\mbox{#1}\hfil}
\newenvironment{namelist}[1]{%
\begin{list}{}
{
\let\makelabel\namelistlabel
\settowidth{\labelwidth}{#1}
\setlength{\leftmargin}{1.1\labelwidth}
}
}{%
\end{list}}

\newcommand{\K}{\mathcal K}
\newcommand{\inp}[2]{\langle {#1} ,\,{#2} \rangle}
\newcommand{\vspan}[1]{{{\rm\,span}\{ #1 \}}}
\newcommand{\R} {{\mathbb{R}}}

\newcommand{\B} {{\mathcal{B}}}
\newcommand{\C} {{\mathbb{C}}}
\newcommand{\N} {{\mathbb{N}}}
\newcommand{\Q} {{\mathbb{Q}}}
\newcommand{\LL} {{\mathbb{L}}}
\newcommand{\Z} {{\mathbb{Z}}}

\title{Zabreiko's Lemma with Bicomplex and hyperbolic scalars and its applications}
\author{Akshay S. RANE\footnote{Department of Mathematics, Institute of Chemical Technology, Nathalal Parekh Marg, Matunga, Mumbai 400 019, India, email :  as.rane@ictmumbai.edu.in}  ~ and ~ Mandar THATTE\footnote{Department of Mathematics, Institute of Chemical Technology, Nathalal Parekh Marg, Matunga, Mumbai 400 019, India, email :  007thattecharlie@gmail.com } 
\hspace {1mm}
}
\date{ }
\maketitle

\begin{abstract}
 In this paper, we shall consider  the notion of hyperbolic semi norm which on a module $X$ to set of all positive hyperbolic numbers. We shall prove  the  characterization of continuity of hyperbolic semi norm in this setup. We shall prove Zabreiko's lemma when $X$ is a F, $\mathbb{BC}$ module, where $\mathbb{BC}$ denotes the set of Bi complex numbers.(analogous to completeness). This lemma shall be used to prove the fundamental theorems of functional analysis like the Closed Graph Theorem, Open mapping Theorem, Uniform Boundedness principle.
\end{abstract}

\noindent
Key Words : Bicomplex modules, Hyperbolic numbers, Zabreiko's Lemma, Closed Graph Theorem, Open mapping Theorem, Uniform Boundedness principle.

\smallskip
\noindent
AMS  subject classification : 46B20, 46C05, 46C15,46B99,46C99
\newpage
\newpage


\setcounter{equation}{0}
\section{Introduction}

\noindent 
Bi complex numbers are being studied for a long time and have been introduced in \cite{ALSS}. There have been several efforts in generalizing the results of functional analysis in the case of modules over Bi complex numbers. In \cite{ALSS}, they he have  discussed basics of Bicomplex numbers, hyperbolic numbers.They also have introduced the notion of hyperbolic norms on modules. In \cite{CSS}, bi omplex holomorphic calculus is developed whereas details about Banach Algebras over bi complex numbers can be found in \cite{EPR}.In \cite{HAR} \cite{KKR}, they have proved the fundamental theorems of functional analysis for bi complex and hyperbolic numbers. It is well known from the classic paper of \cite{Zab} that the Zabreiko's Lemma enables us to prove the major theorems of functional analysis. In this paper, we shall prove the Zabreiko's lemma for Bi Complex numbers  and as an application, we shall prove the fundamental theorems such as the closed graph theorem, Uniform bounded principle, Open mapping theorem in the context of Bi complex numbers.
$$\mathbb{BC} := \lbrace Z =z_1+ z_2 \; j | z_1, z_2 \in \mathbf C (i)  \rbrace $$
where $i,j$ are such that $ij=ji, i^2=j^2=-1.$
The set $\mathbb{BC} $ forms a Ring under the usual addition and multiplication. 
The product of imaginary units $i$ and $j$ defines a hyperbolic unit $k$ such that $k^2=1.$ The product of all units is commutative and satisfies $$ ij=k \;, ik=-j \; jk=-i.$$
The set $\mathbb{D}$ of hyperbolic numbers is defined as 
$$ \mathbb{D} = \lbrace \alpha= \beta_1 + k \beta_2 | \beta_1,\beta_2 \in \mathbb{R} \rbrace .$$ The set $D$ is a ring and a module over itself.
The hyperbolic numbers $e_1$ and $e_2$ are defined as
$$ e_1= \frac{1+k}{2},\; e_2=\frac{1-k}{2}$$ are linearly independent in the $\mathrm{C(i)}$ vector space  $ \mathbb{B}\mathbb{C}$ and satisfy the following properties 
$$e_1^2=e_1, e_2^2=e_2,e_1+e_2=0,e_1e_2=0.$$
Any bi complex number $Z=w_1+jw_2$ can be uniquely written as
$$ Z= e_1 z_1 +e_2z_2$$ where $z_1=w_1-iw_2$ and $z_2=w_1+iw_2$ are elements of $ \mathrm{C(i)}$. 
The hyperbolic-valued or $\mathbb{D}-$ valued norm $|Z|_k$ of a bicomplex number $Z= e_1z_1+e_2z_2$ is defined as 
$$ |Z|_k= e_1 |z_1|+ e_2|z_2|.$$
A hyperbolic number $$ \alpha=e_1 \alpha_1 +e_2 \alpha_2$$
where $ \alpha_1= \beta_1+ \beta_2$ and $ \alpha_2= \beta_1- \beta_2$ are real numbers is said that positive hyperbolic number if $ \alpha_1\geq 0 $ and $ \alpha_2 \geq 0.$ $\alpha$ is strictly positive if $\alpha_1>0$ and $\alpha_2>0.$ 
Thus the set of positive hyperbolic numbers is given by
$$ \mathbb{D}^+ = \lbrace \alpha=e_1  \alpha_1 + e_2 \alpha_2 | \alpha_1, \alpha_2 \geq 0\rbrace $$
\noindent { \bf Remark }
Throughout this paper $\leq$ defines a partial order on $\mathbb{D}$ that is $ \alpha \leq \beta $ if $ \beta - \alpha \in \mathrm{D}^+.$ Also note that if  $\alpha \in \mathbb{D}^+$, then
$$ |\alpha|_k= |\alpha_1| e_1 + |\alpha_2|e_2 =\alpha_1 e_1 + \alpha_2 e_2 = \alpha.$$
Also note that $ \alpha_1 e_1 + \alpha_2 e_2 \in \mathbb{D}^+$ is strictly positive then $\alpha_1 $ and $ \alpha_2$ are strictly positive real numbers and $$ \frac{1}{\alpha_1 e_1 + \alpha_2 e_2  }= \frac{1}{\alpha_1} e_1+ \frac{1}{\alpha_2} e_2$$ is also in $ \mathbb{D}^+$ and strictly positive. 
$\mathrm{D}$ valued norm
Let $X$ be a bi complex module. A map $\|.\|_D: X \rightarrow \mathrm{D}^+ $ is said to be hyperbolic norm if it satisfies the following :\\
(a) $ \|x\|_\mathrm{D}=0$ iff $x=0$ \\
(b) $ \|\mu x\|_\mathrm{D}=|\mu|_K \|x\|_D$ for all $x\in X$ and for all $ \mu \in \mathbb{B}\mathbb{C} $\\
(c) $ \|x+y\|_D < \|x\|_D +\|y\|_D$ for all $ x,y \in X.$

\setcounter{equation}{0}
\section{Main Results}
\noindent We now define the notion of a hyperbolic semi norm. 
Let $X$ be a bi complex module. A map $p_D: X \rightarrow \mathbb{D}^+ $ is said to be hyperbolic semi-norm on $X$ if it satisfies the following :\\
(a) $ p_\mathbb{D}(\mu x)=|\mu|_K p_\mathbb D(x)$ for all $x\in X$ and for all $ \mu \in \mathbb{B}\mathbb{C} $\\
(c) $ p_\mathbb D(x+y) < p_\mathbb D(x) +p_\mathbb D(y) $ for all $x, y \in X.$

\noindent 
Let $X$ and $Y$ be bi complex modules. A map $T: X \rightarrow Y$ is said to be linear if for all $x,y \in X$ $T(x+y)=T(x)+T(y)$ and $T( \mu x) = \mu T(x) $ for $ \mu \in\mathbb{B}\mathbb{C} .$
Let $ \|.\|_{\mathbb D,X} $ and $ \|.\|_{\mathbb D,Y} $ be hyperbolic norms on $X$ and $Y$ respectively. The operator $T$ is said to be $\mathbb{D} $ bounded if there exists $ M \in \mathrm{\mathbb D^+}$ such that $\|T(x)\|_{\mathbb D,Y} \leq M \|x\|_{\mathbb D,X}$.
Define $$ p_\mathbb{D}(x) =\|Tx\|_{\mathbb D,Y}$$
$p$ defines a hyperbolic semi norm on a bi complex module $X.$ 
A hyperbolic semi norm is clearly finitely sub additive. We define $\displaystyle  \sum_{n=1}^ \infty x_n$.
Consider the partial sum $ \displaystyle  s_n= \sum_{k=1}^n x_k$. We say that the above series has a sum provided $s_n$ converges i.e. if there exists $x\in X$ such that $\|s_n -x\|_\mathbb D \rightarrow 0.$
$\displaystyle  \sum_{k=1}^ \infty x_k$ is said to be absolutely summable if $ \displaystyle \sum_{k=1}^ \infty \|x_k\|_\mathbb D $ is convergent.

Let $X$ be a $F-\mathbb{B}\mathbb{C}$ module, which means that every cauchy sequence with respect to the hyperbolic norm converges. Then every absolutely summable series is convergent.
Given that $ \displaystyle \sum_{k=1}^\infty \|x_k\|_\mathbb D$ converges.
Let $$ s_n= \sum_{k=1}^n x_k,\; s_m= \sum_{k=1}^m x_k,\;n>m$$
So $$s_n-s_m = \sum_{k=m+1}^n x_k$$
As a consequence
$$ \|s_n-s_m \|_\mathbb D \leq \|  \sum_{k=m+1}^n x_k\|_\mathbb D \leq \sum_{k=m+1}^n \| x_k \|_\mathbb D \rightarrow 0$$ as $ n,m \rightarrow \infty .$
This implies $s_n$ Cauchy and hence convergent.
For $ \delta \in \mathbb{D}^+ $ note that
$$ \delta = \alpha_1 e_1 + \alpha_2 e_2,$$ where
$\alpha_1, \alpha_2 \geq 0,$
so that 
$$|\delta|_k = |\alpha_1| e_1 + |\alpha_2|e_2 = \alpha_1 e_1 + \alpha_2 e_2 = \delta $$
\noindent 
We shall now prove the Baire Category theorem which shall be used in the main result. 
\begin{theorem} Let X be a F-$\mathbb{B}\mathbb{C}$ module. Let $U_{n}$ be open dense subsets of X for n $\in \mathbb{N}$. Then $\bigcap_{n}U_{n}$ is dense in X.
\end{theorem}

\begin{proof}Let U=$\bigcap_{n}U_{n}.$ Let $x \in X$ and $r>0$. Since $U_{1}$ is open,dense and B(x,r) is open, there exists $x \in B(x,r)\cap U_{1}.$ Since $B(x,r) \cap U_{1}$ is open, choose $0<r<\frac{1}{2}$ such that $B[x_{1},r_{1}] \subset B(x,r) \cap U_{1}.$ Similarly $B(x_{1},r_{1})$ is open and $U_{2}$ is open dense , so,  there exists $x_{2} \in B(x_{1},r_{1}) \cap U_{2}$.Choose $0<r_{2}<(\frac{1}{2})^2$ such that $B[x_{2},r_{2}] \subset B(x_{1},r_{1}) \cap U_{2}.$\\ Proceeding in this way for $ n \in \mathbb{N}$, we get, $x_{n} \in X$ and $r_{n}$ such that $B[x_{n},r_{n}] \subset B(x_{n-1},r_{n-1})$ for $0<r_{n}< \frac{1}{2^{n}}.$\\
for $m \leq n$, $d_{D}(x_{m},x_{n}) \leq d_{D}(x_{n},x_{n-1})+...d_{D}(x_{m+1},x_{m})$ $\leq \sum_{k=m}^{n} (\frac{1}{2})^{k}.$ Since this series converges, $x_{n}$ is Cauchy. Now as X is F-$\mathbb{B}\mathbb{C}$ module, X is a complete metric space. As $x_{n} \in X$, there exists $x_{0} \in X$ such that $x_{n} \rightarrow x_{0}$. As $x_{0}$ is the limit of the sequence of $(x_{n})_{n \geq k}$ in the closed set $B[x_{k},r_{k}]$ we get, $x_{0} \in B[x_{k},r_{k}] \subset B(x_{k-1},r_{k-1}) \cap U_{k}$ for all $k \in \mathbb{N}$. Hence $x_{0} \in B(x,r) \cap U.$ Hence proved.
\end{proof} 
\section{Zabreiko's Lemma}
In this section we shall first prove the equivalent characterization of the continuity of the semi norm and show that in an $F$ space a countabely sub additive semi norm is continuous which is the Zabreiko's Lemma.
\begin{lemma}
 Let X be F-$\mathbb{B}\mathbb{C}$ module.let $p_{\mathbb D}$ be a hyperbolic semi norm on X. $p_{\mathbb D}$ is said to be continuous on X iff there exists $\alpha \in \mathbb D^{+}$ such that $p_{\mathbb D}(x) \leq \alpha \| x\|_{\mathbb D}$ for all $x \in X.$
\end{lemma}

\begin{proof}
	 Suppose  there exists $\alpha \in \mathbb D^{+}$ such that $p_{\mathbb D} \leq \alpha \| x\|_{\mathbb D}$
	We want to show that $p_{\mathbb D}$ is continuous. Let $x_{n} \in X$ , $x_{n} \rightarrow x.$ Consider, $$p_{\mathbb D}(x_{n}) =p_{\mathbb D}(x_{n}-x+x)$$ $$ \leq p_{\mathbb D}(x_{n}-x) + p_{\mathbb D}(x)$$ $$ p_{\mathbb D}(x_{n})-p_{\mathbb D}(x) \leq p_{\mathbb D}(x_{n}-x).$$ Similarly we get,$$ p_{\mathbb D}(x)-p_{\mathbb D}(x_{n}) \leq p_{\mathbb D}(x_{n}-x)$$ so we get,$$ |p_{\mathbb D}(x_{n})-p_{\mathbb D}(x)| \leq p_{\mathbb D}(x_{n}-x)$$ we know that $|z|_{k} \leq \sqrt{2} |z|$ See \cite{ALSS}. Thus, $$|p_{\mathbb D}(x_{n})-p_{D}(x)|_{k} \leq p_{\mathbb D}(x_{n}-x)$$ since $p_{\mathbb D}(x) \leq \alpha \|x\|_{\mathbb D},$ $|p_{\mathbb D}(x_{n})-p_{\mathbb D}(x)|_{k} \leq \alpha \| x_{n}-x\|_{\mathbb D}$ As, $x_{n} \rightarrow x$ we have, $x_{n}-x \rightarrow 0$ therefore, $$|p_{\mathbb D}(x_{n})-p_{\mathbb D}(x)|_{k} \rightarrow 0$$ so, $p_{\mathbb D}(x_{n}) \rightarrow p_{\mathbb D}(x),$ hence $p_{\mathbb D}$ is continuous at x.\\
	Conversely, $p_{\mathbb D}$ is continuous. we want to prove that there exists $\alpha \in\mathbb  D^{+}$ such that $p_{\mathbb D}(x) \leq \alpha \| x\|_{\mathbb D}.$ Suppose for all $\alpha \in \mathbb D^{+}$, there exists $x_{n} \in X$ such that $p_{\mathbb D}(x_{n}) > \alpha \| x_{n}\|_{\mathbb D}$ i.e $$p_{\mathbb D}(x_{n}) - \alpha \| x_{n}\| \in \mathbb D^{+} $$ if $p_{\mathbb D}(x_{n}) =0$ then we get a contradiction. So $p_{\mathbb D}(x_{n}) \neq 0.$ $$ \frac{p_{\mathbb D}(x_{n})}{n \| x_{n}\|_{\mathbb D}} > 1$$ for $n \in \mathbb{N}.$ Let $u_{n}=\frac{x_{n}}{n \| x_{n}\|_{\mathbb D}}$ we observe that as $ n \rightarrow \infty$ , $\|u_{n}\|_{\mathbb D} \rightarrow 0.$ Hence $u_{n} \rightarrow 0$ Now as $p_{\mathbb D}$ is continuous , $p_{\mathbb D}(u_{n}) \rightarrow p_{\mathbb D}(0)=0$ But $p_{\mathbb D} >1$, this is a contradiction, so our assumption was wrong. Hence proved.
\end{proof}\noindent
Next result shows that continuous hyperbolic semi norm is countably sub additive. That is 
$$ p_{\mathbb D}(\sum_{k=1}^{\infty} x_{k}) \leq \sum_{k=1}^{\infty} p_{\mathbb D}(x_{k}).$$ 
\begin{lemma}
	If $p_{\mathbb D}$ is a continuous hyperbolic semi norm on X, a F-$\mathbb{B}\mathbb{C}$ module then $p_{\mathbb D}$ is countabley sub additive.
\end{lemma}

\begin{proof}
	Let $$x=\sum_{n=1}^{\infty} x_{n}$$ be a convergent in X. Assume, $$S_{n}= \sum_{k=1}^{\infty} x_{k}$$ so, $S_{n} \rightarrow x$ as $n \rightarrow \infty$. Since $p_{\mathbb D}$ is continuous, $p_{\mathbb D}(S_{n}) \rightarrow p_{\mathbb D}(x).$ $$ p_{\mathbb D}(x) = \limsup (p_{\mathbb D}(S_{n}))$$ $$ = \limsup (p_{\mathbb D} (\sum_{k=1}^{n} x_{k}))$$ so  $$p_{\mathbb D}(x) \leq \sum_{k=1}^{\infty} p_{\mathbb D}(x_{k})$$ hence we have, $$ p_{\mathbb D}(\sum_{k=1}^{\infty} x_{k}) \leq \sum_{k=1}^{\infty} p_{\mathbb D}(x_{k}).$$ 
\end{proof}

\begin{theorem}
Let $p_\mathbb D$ be a hyperbolic semi norm on a hyperbolic normed space $(X,\|\|_\mathbb D)$ and for $\alpha >0$ consider
$$ V_\alpha = \lbrace x \in X : p_\mathbb D(x) \leq \alpha \rbrace .$$ Suppose there is $a \in X$ and $r>0$ such that
$$ B[a,r] \subset \overline{V_\alpha}$$
Then for every $\delta >0$ we have
$$ B[0,\delta r) \subset \overline{V_ {\delta\alpha}}$$
\end{theorem}
\begin{proof}
Let $\delta=1$ . We shall show that 
$$ B[0, r] \subset \overline{V_ {\alpha}}$$
Let $ x \; B[0, r]$ so that
$$\|x\|_D \leq r$$
Since $$ \|(x+a)-a\|_D \leq r,\; \|(-x+a) -a\|_D \leq r.$$
So we have sequences $u_n$ and $v_n$ such that $u_n \rightarrow x+a$ and $ v_n \rightarrow -x+a.$
Let $x_n= \frac{u_n-v_n}{2}.$
Since $p_\mathbb D$ is a hyperbolic semi norm 
$$ p_\mathbb D(x_n) \leq \frac{ p_\mathbb D(u_n)+p_\mathbb D(v_n)}{2}\leq \alpha  $$ as $|\frac{1}{2}|_k=\frac{1}{2}$
and $$ x_n \rightarrow \frac{[(x+a)-(-x+a)]}{2} =x.$$
Hence $ x \in \overline{V_\alpha}.$
For any $ \delta >0$  and $x \in X$ with $\|x\|_D\leq \delta r.$
Let $y=\frac{x}{\delta}.$ Then
$$ \|y\|_D =\frac{1}{|\delta|_k}\|x\|= \frac{1}{\delta}\|x\|_D\leq r $$
By the earlier proof, $ y \in \overline{V_\alpha}$ So there is a sequence $y_n \in V_\alpha $ such that $ y_n \rightarrow y$
Hence $p(y_n) \leq \alpha. $
So $$|\delta|_kp_\mathbb D(y_n)= p_\mathbb D( \delta y_n)=\delta p_\mathbb D(y_n) \leq \delta \alpha, $$
and $ \delta y_n \rightarrow \delta y= x.$ So $x\in \overline{V_{\delta \alpha}}.$
\end{proof} 
\begin{theorem}
	Let $p_\mathbb D$ be a countabely subadditive hyperbolic seminorm on an $F$ $ \mathbb{B}\mathbb{C}$ $X.$ Then $p_\mathbb D$ is continuous on $X.$ 
\end{theorem}
\begin{proof}
For $n=1,2,\ldots,$
$$ V_n= \lbrace x \in X : p_\mathbb D(x) \leq n \rbrace.$$
Then $$ X= \bigcup_{n=1}^\infty V_n = \bigcup_{n=1}^\infty \overline{V_n}.  $$
Hence $$ \bigcap_{n=1}^\infty (\overline{V_n})^c = \phi$$
By the Baire Category theorem, we know that in a complete metric space, interesection of dense open sets is open. There exist atleast open open set say $(\overline{V_m})^c$ which is not dense. This implies there exist $a \in X $ such that $ a$ does not belong to the closure of $(\overline{V_m})^c$. That means that there exists $r>0$ such that $$B[a,r] \bigcap (\overline{V_m})^c=\phi$$
We have $$ B[a,r] \subset \overline{V_m}.$$
For this purpose, we shall show that 
$$ p_\mathbb D(x) \leq (m/r) \|x\|_D + \epsilon$$ for every $ \epsilon \in \mathrm{D}^+$ strictly positive and for all $ x \in X.$
Define $\epsilon_0 = \frac{\|x\|_D}{r}$ and write
$$ \epsilon= \sum_{k=1}^\infty \frac{\epsilon}{2^k}=m\sum_{k=1}^\infty \frac{\epsilon}{m 2^k} = m \sum_{k=1}^\infty \epsilon_k$$
Since $\|x\|_\mathbb D= \epsilon_0 r$ we have $ x \in B[0, \epsilon r]$. by the above result we have
$$ x \in \overline{V_{\epsilon_0 m}}.$$
This implies
$$ B[x,\epsilon r] \bigcap V_{\epsilon_0 m} \neq \phi .$$
So there exist $x_1 \in B[x,\epsilon r] \bigcap V_{\epsilon_0 m} . $
So we have $\|x-x_1\|_\mathbb D \leq\epsilon_1 r $ and $ p_\mathbb D(x_1) \leq \epsilon_0 m.$
As a result $$x-x_1 \in B[0,\epsilon_1 r] \subset \overline{V_{\epsilon_1 m}} $$
Let $u_1=x-x_1$ so that $x=x_1+u_1.$
$B[u_1,\epsilon_{2}r] \bigcap V_{\epsilon_1 m} \neq \phi $ There exist $x_2 \in B[u_1,\epsilon_{2}r] \bigcap V_{\epsilon_1 m} .$
So we have $$ \|u_1-x_2\|_\mathbb D \leq \epsilon_{2}r,\; p_\mathbb D(x_2) \leq \epsilon_1 m. $$
Let $ u_2=u_1-x_2$ so that
$$ x=x_1 + u_1 = x_1+x_2+u_2.$$
Continuing this way, we find for each $k=1,2,\ldots,u_k \in X $ such that
$$ p_\mathbb D(x_k) \leq \epsilon_{k-1}m,\;\|u_{k-1}-x_k\|_\mathbb D \leq \epsilon_k r  $$ and $u_k=u_{k-1}-x_k.$ As a result
$$x=x_1+u_1= x_1+ x_2+u_2=x_1+x_2+ \ldots + x_k+u_k.$$
As a consequence
$$ \|x- \sum_{k=1}^n x_k\|_\mathbb D = \|u_n\|_\mathbb D \leq \epsilon_n r = \frac{\epsilon r}{m 2^n}\rightarrow 0 $$ as $ n \rightarrow \infty.$ We see that 
$$ x= \sum_{k=1}^ \infty x_k.$$
By the countable sub additivity of the semi norm 
$$ p_\mathbb D(x) \leq \sum_{k=1}^ \infty p_\mathbb D(x_k) \leq m \sum_{k=1}^\infty \epsilon_{k-1} = m \epsilon_{0} + m \sum_{k=1}^\infty \epsilon_k = \frac{m}{r} \|x\|_D + \epsilon.$$
This shows that $p_\mathbb D(x) \leq \frac{m}{r} \|x\|_\mathbb D.$
\end{proof} 
 \section{Application}
We shall now prove the Closed Graph theorem. 
\begin{definition}
A $\mathbb{B}\mathbb{C}$ linear operator $T: X \rightarrow Y$ is said to be closed if its graph is closed in $ X \times Y$. A $\mathbb{B}\mathbb{C}$-linear operator $T$ is closed whenever $x_{n} \rightarrow x$, $Tx_{n} \rightarrow y$, then $Tx =y. $
\end{definition}
\begin{theorem}
 Let $X$ and $Y$ be F-$\mathbb{B}\mathbb{C}$ modules and $T: X \rightarrow Y$ be a closed $\mathbb{B}\mathbb{C}$-linear map. Then $T$ is continuous.
\end{theorem}

\begin{proof}
	 X, Y are F-$\mathbb{B}\mathbb{C}$ Module. $(X,\| .\|_{\mathbb D,X})$ and $(Y,\| .\|_{\mathbb D,Y})$ are complete hyperbolic normed spaces. $p_{\mathbb D}(x) = \| Tx\|_{\mathbb D,Y}$ is a hyperbolic semi norm. Let $\sum_{n=1}^{\infty} x_{n}$ be  convergent  along with $\sum_{n=1}^{\infty} p_{\mathbb D}(x_{n})$ also summable . i.e $\displaystyle \sum_{n=1}^{\infty} \| Tx_{n}\|_{\mathbb D,Y}$ is convergent. So,  $\displaystyle \sum_{n=1}^{\infty} Tx_{n}$ is absolutely summable. As Y is F-$\mathbb{B}\mathbb{C}$ module, absolutely $\implies$ summable. Hence  $\displaystyle \sum_{n=1}^{\infty} Tx_{n}$ is convergent.let $$ \sum_{n=1}^{\infty} x_{n} = x$$ and $$ \sum_{n=1}^{\infty} Tx_{n}.$$ let $$ S_{n}=\sum_{k=1}^{n} x_{k}$$ so, $S_{n} \rightarrow x$ as $n \rightarrow \infty$ Also, $$ T(s_{n}) = \sum_{k=1}^{n} Tx_{k}$$ $T(S_{n}) \rightarrow y$ as $ n\rightarrow \infty.$ Now as T is closed, $Tx=y.$ hence, $$ \| Tx\|_{\mathbb D,Y} = \| y\|_{\mathbb D,Y}$$ and $$ \| Ts_{n}\|_{\mathbb D,Y} \rightarrow \| y\|_{\mathbb D,Y}$$ so we get, $$ \|Tx\|_{\mathbb D,Y} = \| y\|_{\mathbb D,Y} = \lim_{n\to\infty} \|Ts_{n}\|_{\mathbb D,Y}$$ $$ = \lim_{n\to\infty} \|T(\sum_{k=1}^{n} x_{k})\|_{\mathbb D,Y}$$ $$ \lim_{n\to\infty} \| \sum_{k=1}^{n} Tx_{k}\|_{D,Y}$$ $$ \leq \lim_{n\to\infty} \sum_{k=1}^{n} \|Tx_{k}\|_{\mathbb D,Y}$$ $$ \leq \sum_{n=1}^{\infty} \| Tx_{n}\|_{\mathbb D,Y}$$ Hence we get, $$ p_{\mathbb D}(x)=p_{\mathbb D}(\sum_{n=1}^{\infty} x_{n}) \leq \sum_{n=1}^{\infty} p_{\mathbb D}(x_{k})$$ therefore,$p_{\mathbb D}$ is countabley subadditive. By Zabreiko's lemma $p_{\mathbb D}$ is continuous. hence, there exists $\alpha \in \mathbb D^{+}$ such that $$ \|Tx\|_{\mathbb D,Y} \leq \alpha \| x\|_{\mathbb D,X}.$$ Hence T is continuous. 
\end{proof}
\noindent 
We shall now prove the Uniform boundedness principle.
\begin{theorem} Let X be F-$\mathbb{B}\mathbb{C}$ module and $P_{0}$ be the set of continuous hyperbolic semi norms on X such that the set $ \lbrace p_{D}(x) : p_{D} \in P_{0} \rbrace$ is $\mathbb{D}$-bounded for each $x \in X$, then, there exists $\delta \in D^{+}$ such that $p_{D}(x) \leq \delta \| x\|_{D}$ for all $x \in X$ and for all $p_{D} \in P_{0}.$
	\end{theorem} 
\vspace{0.2cm}

\begin{proof} define $p^{*} : X \rightarrow \mathbb D^{+}$ as $$ p^{*}(x) = \sup\lbrace p_{\mathbb D}(x) , p_{\mathbb D} \in P_{0} \rbrace$$ It is clearly well defined. Now we need to first show that this is a semi norm. So proving the triangle inequality is just trivial. now, $$p^{*}(\alpha x) = \sup \lbrace p_{\mathbb D}(\alpha x) : p_{\mathbb D} \in P_{0} , \alpha \in \mathbb{B}\mathbb{C} \rbrace$$ $$ = \sup \lbrace | \alpha|_{k} p_{\mathbb D}(x) : p_{\mathbb D} \in P_{0} , \alpha \in \mathbb{B}\mathbb{C} \rbrace$$ $$= | \alpha|_{k} \sup \lbrace p_{\mathbb D}(x) : p_{\mathbb D} \in P_{0} , \alpha \in \mathbb{B}\mathbb{C} \rbrace$$ $$ = | \alpha|_{k} p^{*}(x)$$ Thus, $p^{*}$ is a hyperbolic semi norm on X. Now we will show that it is countabley sub additive. Let $\sum_{k=1}^{\infty} x_{k}$ be convergent along with $\sum_{k=1}^{\infty} p^{*}(x_{k}).$ let $$ \sum_{k=1}^{\infty} x_{k} = x$$ $$ p^{*}(x) = \sup \lbrace p_{\mathbb D}(\sum_{k=1}^{\infty} x_{k} : p_{\mathbb D} \in P_{0} \rbrace$$ $$ \leq \ \lbrace \sum_{k=1}^{\infty} p_{\mathbb D}(x_{k}) : p_{\mathbb D} \in P_{0} \rbrace $$ $$ \leq \sum_{k=1}^{\infty} \sup \lbrace p_{D} (x_{k}) : p_{D} \in P_{0} \rbrace.$$ Thus, $$ p^{*}(\sum_{k=1}^{\infty}x_k) \leq \sum_{k=1}^{\infty} p^{*}(x_{k}).$$ Hence, $p^{*}$ is countabley sub additive hyperbolic semi norm on X.By Zabreiko's lemma, $p^{*}$ is continuous. Hence  there exists $\delta \in D^{+}$ , $$p^{*}(x) \leq \delta \| x\|_\mathbb {D}$$ So, we get, $$ p_{\mathbb D}(x) \leq p^{*} \leq \delta \| x\|_{D}$$ for all $x \in X$ and for all $p_{\mathbb D} \in P_{0}.$ Hence proved.
	\end{proof}
 \begin{theorem}
	Let X be a F-$\mathbb{B}\mathbb{C}$ Module. for each $s \in S$, let $(Y_{s},\| .\|_{\mathbb  D,s})$ be a hyperbolic normed space. Let $T_{s}$ be a continuous $\mathbb{B}\mathbb{C}$-linear map. $T_{s} : X \rightarrow Y_{s}$ be such that for each $x \in X,$ $\lbrace \| T_{s}(x)\|_{\mathbb  D} : s\in S \rbrace$ is $\mathbb{D}$-bounded then the set $\lbrace \| T_{s}\|_{\mathbb D} : s \in S \rbrace$ is $\mathbb{D}$-bounded.
\end{theorem}

\begin{proof}
	proof: define $p_{s} : X \rightarrow \mathbb  D^{+}$ as $$ p_{s}(x) = \| T_{s}(x)\|_{\mathbb D}$$ Since $T_{s}$ is continuous $\mathbb{B}\mathbb{C}$-linear map, there exists $\alpha \in D^{+}$ such that, $$ \| T_{s}(x)\|_{\mathbb D} \leq \alpha \|x\|_\mathbb {D}.$$ Thus, $$ p_{s}(x) \leq \alpha \| x\|_{D}$$ let, $P= \lbrace p_{s} : X \rightarrow \mathbb D^{+} : s \in S \rbrace$ is a family of continuous hyperbolic semi norms on X.$$\lbrace p_{s}(x) : s \in S \rbrace = \lbrace \| T_{s}(x)\|_{\mathbb D} : s \in S \rbrace$$ is $\mathbb{D}$-bounded in X. So, by Uniform Bounded Theorem, $$ p_{s}(x) \leq \alpha \| x\|_{\mathbb D}$$ for all $x \in X$ and $s \in S.$ $$ \| T_{s}(x)\|_{\mathbb D} \leq \alpha \| x\|_{\mathbb D}$$ we have, $$ \frac{\| T_{s}(x)\|_{\mathbb D}}{\| x\|_{\mathbb D}} \leq \alpha $$ Hence, $$ \| T_{s}(x)\|_{\mathbb D} \leq \alpha $$ for all $s \in S.$ Hence proved.
\end{proof}\noindent
Now we shall prove the open mapping theorem. We shall require a lemma before we prove the open mapping theorem.

\begin{lemma}
  let	X and Y be hyperbolic normed spaces. $T: X \rightarrow Y$ is $\mathbb{B}\mathbb{C}$-linear map. Then, T is open map iff there exists $\delta \in D^{+}$ such that for all $y \in Y$ , there exists $x \in X$ such that $T(x)=y$ and $\| x\|_{\mathbb D} \leq \delta \|y\|_{\mathbb D}.$
	
\end{lemma}
\begin{proof}
	 Suppose T is an open map. we want to prove that for $y\in Y$ , there exists $x \in X$ such that $T(x)=y$ and $\| x\|_{\mathbb D} \leq \delta \|y\|_{\mathbb D}.$\\
  Consider, $U= \lbrace x \in X : \| x\|_{\mathbb D}<1 \rbrace $ is a open ball hence a open set. so, $T(U)$ is open set in Y. As $0 \in U$, $T(0)=0\in T(U)$ is open. there exists $\delta \in \mathbb D^{+} $, $$B[0,\delta]\subset T(U).$$ let $$ y\in Y, y \neq 0 , \frac{\delta y}{\| y\|_{\mathbb D}} \in B[0,\delta] \subset T(U).$$
  there exists $\widetilde{x} \in U$, $$\frac{\delta y}{\| y\|_{\mathbb D}} = T\widetilde{x}$$ $$ y=\frac{\| y\|_{\mathbb D}}{\delta} T\widetilde{x} = T(\frac{\| y\|_{\mathbb D} \widetilde{x}}){\delta}$$ $$ \| x\|_\mathbb {D} = \| \frac{\| y\|_{\mathbb D} \widetilde{x}}{\delta}\|_{\mathbb D} \leq \frac{\| y\|_{\mathbb D}}{\delta}.$$ As $\| \widetilde{x} \|_{\mathbb D} <1$, we get, $$ \| x\|_{\mathbb D} \leq \frac{\| y\|_{\mathbb D}}{\delta}.$$ As $ \gamma=\frac{1}{\delta} \in \mathbb D^{+}$ we get, $$ \| x\|_{\mathbb D} \leq \gamma \| y\|_{\mathbb D}.$$ 
  Conversely, we have, there exists $\delta \in D^{+}$ such that for every $y \in Y$, there exists $x \in X$ such that $T(x)=y$ and $\| x\|_{D} \leq \delta \| y\|_{D}.$ we want to show that $ T$ is open.\\
  Let $E$  be open in $X$. we want to prove that $T(E)$ is open $ Y$. let $y_{0} \in T(E)$ so, $y_{0} = T(x_{0})$ for some $ x_{0} \in E.$\\
  Since E is open, there exists $\delta \in D^{+}$ such that $B(x_{0},\delta) \subset E.$\\ Claim:$$ B(T(x_{0}),\frac{\delta}{\gamma}) \subset T(E).$$
  let $y\in B(T(x_{0},\frac{\delta}{\gamma})$
   $$ \| y-T(x_{0})\|_{\mathbb D} < \frac{\delta}{\gamma}$$ so, $$ y-T(x_{0}) \in Y$$  there exists $x \in X , T(x) = y-T(x_{0})$ and $\| x\|_{D} \leq \gamma \| y-T(x_{0})\|_{\mathbb D} < \delta$ $$ y=T(x)+T(x_{0})=T(x+x_{0})$$ $$ \| x\|_{D} < \delta , \| x+x_{0}-x{0}\|_{\mathbb D} < \delta$$ so, $$ x+x_{0} \in B(x_{0},\delta)$$ As $x+x_{0} \in E,$  $T(x+x_{0}) \in T(E).$ This implies $y\in T(E).$ Hence, $T(E)$ is open. 
\end{proof}

\begin{theorem}
If $X$ and $Y$ be F-$\mathbb{B}\mathbb{C}$ modules. $T: X \rightarrow Y$ is surjective and closed then $T$ is $\mathbb{D}$-bounded and $T$ is open map.    
\end{theorem}
\begin{proof}
   Closed Graph theorem implies that $ T$ is $\mathbb{D}$-bounded. 
    Consider $q : Y \rightarrow \mathbb  D^{+}$ given by $$ q(y)= \inf \lbrace \| x\|_{\mathbb D} , x \in X , Tx=y \rbrace.$$
    We want to show that $q(y)$ is countably sub additive hyperbolic semi norm. Clearly, $q(y)$ is well defined. Also proving the triangle inequality is trivial. Now if $\alpha =0$, $q(0 y)=0.$ So assume $\alpha \in \mathbb D^{+}$, $$ q(\alpha y) = \inf \lbrace \|  x\|_{D} x\in X , Tx=\alpha y \rbrace$$ $$ = \inf \lbrace \| x\|_{\mathbb D} , x\in X, T(\frac{x}{\alpha})=y \rbrace$$ let $u=\frac{x}{\alpha}$ , so $ x= \alpha u$ we get, $$ q(\alpha y) =\inf \lbrace \| \alpha u\|_{\mathbb D} , u \in X , T(u)=y \rbrace$$ $$ =\inf \lbrace | \alpha|_{k} \| u\|_{\mathbb D} , u\in X , T(u)=y \rbrace $$ $$ =| \alpha|_{k} \lbrace \| u\|_{\mathbb D} , u\in X , T(u)=y \rbrace$$ $$ q(\alpha y) = | \alpha |_{k} q(y).$$ Hence, $q$ is a hyperbolic semi norm.\\ Now we want to show $q$ is countably sub additive. Let $\displaystyle \sum_{k=1}^{\infty} y_{k}$ be a convergent  with $\displaystyle \sum_{k=1}^{\infty} q(y_{k})$ also convergent.$$ q(y_{k}) = \inf \lbrace \| x\|_{D} , x\in X , T(x) = y_{k} \rbrace.$$ Let $ \epsilon \in \mathbb D^{+}$ be strictly positive. For each $k$ , $$ \epsilon = \sum_{k=1}^{\infty} \frac{\epsilon}{2^{k}} = \sum_{k=1}^{\infty} \epsilon_{k}.$$ As $T$ is surjective there exists $x_{k} \in X$ such that $ T(x_{k}) = y_{k}$ and $$ \| x_{k}\|_{\mathbb D} \leq  q(y_{k}) + \epsilon_{k}.$$ As a consequence $$  \sum_{k=1}^{\infty} \| x_{k}\|_{\mathbb D} \leq \sum_{k=1}^{\infty} q(y_{k}) + \sum_{k=1}^{\infty} \epsilon_{k}.$$ Hence comparison test we get, $\sum_{k=1}^{\infty} \| x_{k}\|_{\mathbb D}$ is convergent.
    As absolute summable implies summable since X is F-$\mathbb{B}\mathbb{C}$ module. we get, $$ \sum_{k=1}^{\infty} x_{k}=x$$
    let $$ S_{n}=\sum_{k=1}^{n} x_{k}$$ $S_{n} \rightarrow x$ as $n \rightarrow \infty$ So, $T(S_{n}) \rightarrow T(x)$ (T is continuous by closed graph theorem.) $$ T(\sum_{k=1}^{n} x_{k}) \rightarrow \sum_{k=1}^{n} y_{k} \rightarrow y$$ we have $y=T(x).$ Hence , $$ q(y) \leq \| x\|_{\mathbb D} = \| \sum_{k=1}^{\infty} x_{k}\|_{D} \leq \sum_{k=1}^{\infty} \| x_{k}\|_{\mathbb D} \leq \sum_{k=1}^{\infty}q(y_k) +\sum_{k=1}^{\infty} \epsilon_{k}.$$ Thus, $$ q(\sum_{k=1}^{\infty} y_{k}) \leq \sum_{k=1}^{\infty} q(y_{k}) + \epsilon.$$ Thus $ q$ is countabley sub additive. \\
    As $q$ is countably sub additive hyperbolic semi norm on Y by Zabreiko's lemma q is continuous. so, there exists $\alpha \in D^{+}$ such that $q(y) \leq \alpha \| y\|_{\mathbb D}.$ Hence for all $y \in Y$ , there exists $x \in X$ such that $T(x)=y$ and $\| x\|_{\mathbb D} < \alpha \| y\|_{\mathbb D}.$ Therefore by the previous lemma we get that $ T$ is open.
\end{proof}

\noindent
Acknowledement : The author Akshay S. Rane would like to thank UGC faculty recharge program, India for their support.

\end{document}